\documentclass[a4paper,10pt]{amsart}
\usepackage[utf8]{inputenc}
\usepackage{amsxtra}
\usepackage{amsopn}
\usepackage{amsmath,amsthm,amssymb}
\usepackage{amscd}
\usepackage{amsfonts}
\usepackage{latexsym}
\usepackage{verbatim}
\usepackage{color}

\theoremstyle{plain}
\newtheorem{theorem}{Theorem}[section]
\newtheorem*{theorem*}{Theorem}
\newtheorem*{theorem**}{Theorem A}
\newtheorem*{theorem***}{Theorem B}
\newtheorem{definition}[theorem]{Definition}
\newtheorem{lemma}[theorem]{Lemma}
\newtheorem{prop}[theorem]{Proposition}
\newtheorem{cor}[theorem]{Corollary}
\newtheorem{rem}[theorem]{Remark}
\newtheorem{example}[theorem]{Example}
\newtheorem*{mt*}{Main Theorem}
\newcommand\SU{\hbox{\rm SU}}

\makeatletter
\newenvironment{proofof}[1]{\par
  \pushQED{\qed}%
  \normalfont \topsep6\p@\@plus6\p@\relax
  \trivlist
  \item[\hskip\labelsep
        \bfseries
    Proof of #1\@addpunct{.}]\ignorespaces
}{%
  \popQED\endtrivlist\@endpefalse
}
\makeatother

\newcommand{\del}{\partial}
\newcommand{\delbar}{\overline{\del}}

\title[Some remarks on Hermitian manifolds]{Some remarks  on  Hermitian manifolds satisfying K\"ahler-like conditions}
\author{Anna Fino}
\address[Anna Fino]{Dipartimento di Matematica ``G. Peano'' \\
Universit\`{a} degli studi di Torino \\
Via Carlo Alberto 10\\
10123 Torino, Italy
}
\email{annamaria.fino@unito.it}

\author{Nicoletta Tardini}
\address[Nicoletta Tardini]{Dipartimento di Matematica ``G. Peano'' \\
Universit\`{a} degli studi di Torino \\
Via Carlo Alberto 10\\
10123 Torino, Italy}
\email{nicoletta.tardini@gmail.com}
\curraddr{
Dipartimento di Scienze Matematiche, Fisiche e Informatiche\\
Unit\`a di Matematica e Informatica\\
Universit\`a degli Studi di Parma\\
Parco Area delle Scienze 53/A\\
43124 Parma, Italy
}

\keywords{Bismut connection, curvature, Hermitian metric}
\subjclass[2010]{53C55; 53C05; 22E25; 53C30; 53C44}
\begin{document}

\maketitle

\begin{abstract} We study Hermitian metrics whose Bismut connection  $\nabla^B$ satisfies  the first Bianchi identity  in relation to the SKT condition and the parallelism of the torsion of the Bimut connection. We obtain a characterization of complex surfaces admitting Hermitian metrics whose Bismut connection satisfy the first Bianchi identity and the  condition   $R^B(x,y,z,w)=R^B(Jx,Jy,z,w)$, for every tangent vectors $x,y,z,w$, in terms of Vaisman metrics.  These conditions, also called  Bismut K\"ahler-like,  have been recently studied in \cite{angella-otal-ugarte-villacampa, zhao-zheng, yau-zhao-zheng}.
Using the characterization of SKT almost abelian Lie groups in \cite{arroyo-lafuente},  we construct new examples of Hermitian manifolds satisfying  the  Bismut K\"ahler-like condition. Moreover, we prove   some   results  in relation to  the pluriclosed flow on complex surfaces and  on almost abelian Lie groups. In particular, we show that, if the initial metric has constant scalar curvature, then
the pluriclosed flow preserves the Vaisman condition on complex surfaces.

\end{abstract}

\section{Introduction}

Given a Hermitian manifold $(X,J, g),$  the Bismut connection $\nabla^B$  is the unique connection on $X$ that is Hermitian (i.e.  such that $\nabla^B g =0$ and $\nabla^B J =0$)  and has totally skew-symmetric torsion tensor (cf. \cite{bismut}). If the torsion $3$-form of $\nabla^B$   is closed, the Hermitian metric $g$ is called {\em SKT} or {\em pluriclosed}.
Examples of  SKT  manifolds  are given by Lie groups (and its compact quotients)  endowed  with left-invariant  SKT  Hermitian structures (see for instance \cite{FPS}, \cite{Ugarte},  \cite{FOU}).   A characterization of  almost abelian  Lie groups, i.e. Lie groups whose  Lie algebra has a codimension- one abelian ideal,    admitting  left-invariant SKT metrics   has been recently obtained in \cite{arroyo-lafuente}.

In \cite{yang-zheng-kahler-like}  the authors studied  Hermitian metrics whose Levi-Civita and Chern connection have curvature tensors satisfying  all the symmetry conditions of a K\"ahler metric.
Hermitian metrics with the Bismut  connection being \lq \lq K\"ahler-like", namely, satisfying the first Bianchi identity and the condition $R^B(x,y,z,w)=R^B(Jx,Jy,z,w)$, for every tangent vectors $x,y,z,w$, have been studied,  in \cite{angella-otal-ugarte-villacampa},     investigating this property on  6-dimensional solvmanifolds with holomorphically trivial canonical bundle. 

In \cite{zhao-zheng}  Zhao and Zheng  show that if the curvature tensor of the Bismut connection  satisfies the symmetry conditions  
\begin{equation}\label{condition-skl}
R^B(x,y,z,w) = R^B(z,y,x,w)\,,\quad R^B(x,y,Jz,Jw) = R^B(x,y,z,w)
\end{equation}
for any tangent vectors $x, y, z, w$ in  $X$, then the  Hermitian metric must be SKT.
In \cite{yau-zhao-zheng}  a classification for compact non-K\"ahler Hermitian manifolds satisfying \eqref{condition-skl} in complex dimension $3$ and those with degenerate torsion in higher dimensions is given.\\

An evolution equation of SKT metrics is given by the  pluriclosed flow, introduced by Streets and Tian in \cite{streets-tian}. In recent years the pluriclosed flow has been an active subject of study, and already many regularity and convergence results have been proved in \cite{streets-tian-2, streets-tian-3}.  A natural question is to see  if the Bismut K\"ahler-like condition is preserved by the flow.

In this paper we study  Hermitian metrics whose Bismut connection  $\nabla^B$ satisfies  the first Bianchi identity in relation to the SKT condition and the parallelism of the torsion of the Bismut connection.
In particular, as a consequence of Theorem \ref{thm:skt-bianchi-then-parallel} and Theorem \ref{thm:parallel-bianchi-iff-skt} we show that, if  $X$ is  a complex manifold and $g$ is  a Hermitian metric such that the Bismut connection satisfies the first Bianchi identity, then, 
$$
\nabla^BT^B=0 \quad\iff\quad g \text{ is SKT.}
$$

Moreover, we show in Proposition \ref{prop:deformations} that the existence of these metrics is not open under small deformations of the complex structure.\\
Specializing to complex dimension $2$ we are able to characterize these metrics as follows 
\begin{theorem**} \label{thm:characterization-vaisman}
Let $X$ be a complex surface and $g$ be a Hermitian metric. Then, $g$ is Vaisman if and only if $g$ is a
SKT metric and the Bismut connection satisfies the first Bianchi identity. 
\end{theorem**}
Where we recall that a Vaisman metric $\omega$ on a complex manifold $X$ is a Hermitian metric satisfying  $d\omega=\theta\wedge\omega$ for some $d$-closed $1$-form $\theta$ with $\nabla^{LC}\theta=0$.

In \cite{belgun-threefolds} a generalization of Vaisman metrics, called metrics with Lee potential, is introduced, and in subsection \ref{subs:LP} we study this condition in relation with the SKT condition.

In Section \ref{sect:almost-abelian} we  construct  new examples of Hermitian manifolds  satisfying  the Bismut K\"ahler-like condition, using the characterization of   SKT simply connected  almost abelian Lie groups. In order to do this we compute explicitly the components of the Bismut connection.
Moreover, we give conditions on the structure equations of almost abelian Lie groups in order to have K\"ahler and flat Hermitian metrics.

In the last Section we study  the   Bismut K\"ahler-like condition  in relation to the pluriclosed flow, discussing its behavior on complex dimension $2$ and  on almost abelian  Lie groups.
In particular, we prove the following 
\begin{theorem***}\label{theorem-b}
Let $X$ be a compact complex surface admitting a Vaisman metric $\omega_0$ with constant scalar curvature, then the pluriclosed flow starting with $\omega_0$ preserves the Vaisman condition.
\end{theorem***}

\medskip
\noindent{\sl Acknowledgments.}  The authors would like to thank Luigi Vezzoni and Mihaela Pilca for useful discussions and suggestions.
The authors would like to thank also Quanting Zhao and Fangyang Zheng for useful comments.
The authors are grateful to Jeffrey Streets for pointing out an inaccuracy in the first version of Theorem B.
The authors are supported by
Project PRIN 2017 ``Real and complex manifolds: Topology, Geometry and Holomorphic Dynamics'',
by project SIR 2014 AnHyC ``Analytic aspects in complex and hypercomplex geometry'' code RBSI14DYEB, and by GNSAGA of INdAM.

\section{Preliminaries}

Let $(X,J)$ be a complex manifold of complex dimension $n$ and let $g$ be a Hermitian metric on $X$ with associated fundamental form $\omega(\cdot\,,\cdot)=g(\cdot,J\,\cdot)$\,. An affine connection is called Hermitian if it preserves the metric $g$ and the complex structure $J$. In particular, Gauduchon in \cite{gauduchon-bumi} proved that there exists an affine line $\left\lbrace\nabla^t\right\rbrace_{t\in\mathbb{R}}$ of canonical Hermitian connections, passing through the \emph{Chern connection} and the \emph{Bismut connection}; these connections are completely determined by their torsion. 
Let $\nabla$ be a Hermitian connection and $T(x,y)=\nabla_xy-\nabla_yx-[x,y]$ be its torsion, we denote with the same symbol
$$
T(x,y,z):=g(T(x,y),z).
$$
Then the Chern connection $\nabla^{Ch}$ is the unique Hermitian connection whose torsion has trivial $(1,1)$-component and the Bismut connection (also called Strominger connection) $\nabla^B$ is the unique Hermitian connection with totally skew-symmetric torsion. In particular, the torsion of the Bismut connection satisfies
$$
T^B(x,y,z)=d^c\omega(x,y,z)
$$
where $d^c=-J^{-1}dJ$.\\
A Hermitian metric $\omega$ is called \emph{strong K\"ahler with torsion} (\emph{SKT} for brevity) or \emph{pluriclosed} if $T^B$ is a closed $3$-form, namely $dT^B=0$, or equivalently $dd^c\omega=0$.\\
Recall that the trace of the torsion of the Chern connection is equal to the Lee form of $\omega$ (cf. \cite{gauduchon}), that is the $1$-form defined by
$$
\theta=Jd^*\omega
$$
where $d^*$ is the adjoint of the exterior derivative $d$ with respect to $\omega$, or equivalently $\theta$ is the unique $1$-form satisfying
$$
d\omega^{n-1}=\theta\wedge\omega^{n-1}\,.
$$
A Hermitian metric $\omega$ is called \emph{Gauduchon} if
$dd^c\omega^{n-1}=0$, or equivalently
$d^*\theta=0$. In particular, in dimension $2$ Gauduchon and SKT metrics coincide.
We recall the following
\begin{definition}
A Hermitian metric $\omega$ on $X$ is called \emph{locally conformally K\"ahler} (\emph{lck} for brevity) if
$$
d\omega=\alpha\wedge\omega
$$
where $\alpha$ is a $d$-closed $1$-form. In particular, $\alpha=\frac{1}{n-1}\theta$ and $\theta$ is $d$-closed.\\
A locally conformally K\"ahler metric is called \emph{Vaisman} if the Lee form is parallel with respect to the Levi-Civita connection $\nabla^{LC}$, namely
$$
\nabla^{LC}\theta=0\,.
$$
\end{definition}
In particular, it is immediate to see that Vaisman metrics are Gauduchon and the norm of the Lee form $|\theta|$ with respect to $\omega$ is constant.\\

The Chern and Bismut connections connections are related to the Levi-Civita connection $\nabla^{LC}$ by
$$
\begin{aligned}
g(\nabla^B_xy,z)&=g(\nabla^{LC}_xy,z)+\frac{1}{2}d^c\omega(x,y,z)\,,\\
g(\nabla^{Ch}_xy,z)&=g(\nabla^{LC}_xy,z)+\frac{1}{2}d\omega(Jx,y,z)\,.
\end{aligned}
$$
If we denote with
$$
\begin{aligned}
R(x,y)z &=\nabla_x\nabla_yz-\nabla_y\nabla_xz-\nabla_{[x,y]}z\,,\\
R(x,y,z,u) &=g(R(x,y)z,u)\,,
\end{aligned}
$$
the curvature tensor of type $(1,3)$ and $(0,4)$, respectively, of a connection $\nabla$ then we have the following identities involving the torsion and the curvature of the Bismut connection (cf. \cite{ivanov-papadopoulos}) which will be useful in the following (cf.  \cite[Formulas (3.20), (3.21)]{ivanov-papadopoulos})
\begin{equation}\label{formula:3-20-ivanov-papadopoulos}
\begin{aligned}
dT^B(x,y,z,u)=&\,\sigma_{x,y,z}\left\lbrace(\nabla^B_xT^B)(y,z,u)+
2g\left(T^B(x,y),T^B(z,u)\right)\right\rbrace\\
& -(\nabla^B_uT^B)(x,y,z)\,,
\end{aligned}
\end{equation}
\begin{equation}\label{formula:3-21-ivanov-papadopoulos}
\begin{aligned}
\sigma_{x,y,z}R^B(x,y,z,u)=&\, dT^B(x,y,z,u)+(\nabla^B_uT^B)(x,y,z)\\
& -\sigma_{x,y,z}g\left(T^B(x,y),T^B(z,u)\right)\,.
\end{aligned}
\end{equation}
In particular, this shows that the Bismut connection does not satisfy the first Bianchi identity in general. We recall the following definition (cf. \cite[Definition 4]{angella-otal-ugarte-villacampa}, \cite{yang-zheng-kahler-like}, \cite{zhao-zheng})

\begin{definition}
The Bismut connection $\nabla^B$ is called \emph{K\"ahler-like} if it satisfies the first  Bianchi identity
\begin{equation}\label{first-bianchi}
\sigma_{x,y,z}R^B(x,y,z)=0
\end{equation}
and the type condition
\begin{equation}\label{J-invariance}
R^B(x,y,z,w)=R^B(Jx,Jy,z,w)\,.
\end{equation}
\end{definition}

In \cite{angella-otal-ugarte-villacampa} the authors study these two conditions for the canonical connections considered by Gauduchon on $6$-dimensional solvmanifolds with invariant complex structures, trivial canonical bundle and  invariant Hermitian metrics.\\
In \cite{zhao-zheng}  Zhao and Zheng  show that if the curvature tensor of the Bismut connection  satisfies the symmetry conditions  $$
R^B(x,y,z,w) = R^B(z,y,x,w)\,,\quad R^B(x,y,Jz,Jw) = R^B(x,y,z,w)
$$
for any tangent vectors $x, y, z, w$ in  $X$, then the  Hermitian metric must be SKT.
We will show that the  previous condition for the  curvature of the Bismut connection to be symmetric when the first and the third position are interchanged is stronger than first Bianchi identity.
In fact, in a private communication Zhao and Zheng showed that those two conditions are equivalent to the vanishing of the curvature $R^B$.

\section{Bismut K\"aher-like condition} 

In this section we investigate the properties (\ref{first-bianchi}) and (\ref{J-invariance}) in the definition of Bismut K\"ahler like Hermitian metrics.
In particular, we focus on the relations  between   the  first Bianchi identity for the Bismut connection, the SKT condition and the parallelism of the  torsion of the Bismut connection, specializing then our considerations to dimension $2$, giving a characterization for Vaisman metrics.

\begin{theorem}\label{thm:skt-bianchi-then-parallel}
Let $X$ be a complex manifold with a SKT Hermitian metric $g$ such that the Bismut connection satisfies the first Bianchi identity. Then,
$$
\nabla^BT^B=0.
$$
\end{theorem}
\begin{proof}
By hypothesis, $dT^B=0$ and the  first Bianchi-identity holds, hence by Formula (\ref{formula:3-21-ivanov-papadopoulos}) one has that for any tangent vectors $x,y,z,u$,
$$
(\nabla_uT^B)(x,y,z)=\sigma_{x,y,z}g(T^B(x,y),T^B(z,u))
$$
and by Formula (\ref{formula:3-20-ivanov-papadopoulos})
$$
\sigma_{x,y,z}(\nabla^B_xT^B)(y,z,u)+\sigma_{x,y,z}2g(T^B(x,y),T^B(z,u))-
(\nabla_u^BT^B)(x,y,z)=0\,.
$$
Hence, one gets
$$
\sigma_{x,y,z}(\nabla^B_xT^B)(y,z,u)=-(\nabla^B_uT^B)(x,y,z).
$$
Now notice that both sides of the equality are tensorial and on the left hand side the expression is symmetric in $x,y,z$ while on the right hand side is antisymmetric in $x,y,z$. Therefore in order to be equal they must vanish. Hence, $\nabla^BT^B=0$.
\end{proof}
\begin{theorem}\label{thm:parallel-bianchi-iff-skt}
Let $X$ be a complex manifold and let $g$ be a Hermitian metric such that $\nabla^BT^B=0$. Then, the Bismut connection satisfies the  first Bianchi identity if and only if $g$ is SKT.
\end{theorem}
\begin{proof}
If $\nabla^BT^B=0$ then by Formula (\ref{formula:3-20-ivanov-papadopoulos})
$$
dT^B(x,y,z,u)=\sigma_{x,y,z}2g(T^B(x,y),T^B(z,u))
$$
and so by Formula (\ref{formula:3-21-ivanov-papadopoulos})
$$
\sigma_{x,y,z}R^B(x,y,z,u)=dT^B(x,y,z,u)-\sigma_{x,y,z}g(T^B(x,y),T^B(z,u))=
\frac{1}{2}dT^B(x,y,z,u)\,.
$$
Therefore, the Bismut connection satisfies the first Bianchi identity if and only if $g$ is SKT.
\end{proof}
Hence, by putting together Theorem \ref{thm:skt-bianchi-then-parallel} and Theorem \ref{thm:parallel-bianchi-iff-skt} we obtain the following
\begin{cor}\label{cor:bianchi-parallel-iff-skt}
Let $X$ be a complex manifold and let $g$ be a Hermitian metric such that the Bismut connection satisfies the first Bianchi identity. Then, 
$$
\nabla^BT^B=0 \quad\iff\quad g \text{ is SKT.}
$$
\end{cor}

Now we consider some relations with respect to the Levi-Civita connection
\begin{theorem}\label{thm:bianchi-skt-then-levicivita-parallel}
Let $X$ be a complex manifold and $g$ a Hermitian metric.
If the Bismut connection satisfies the first Bianchi identity and $g$ is SKT then
$$
\nabla^{LC}T^B=0.
$$
\end{theorem}
\begin{proof}
If the Bismut connection satisfies the first Bianchi identity and $g$ is SKT ($dT^B=0$), then by Theorem \ref{thm:skt-bianchi-then-parallel} we have
$\nabla^BT^B=0$ and so by Formula (\ref{formula:3-21-ivanov-papadopoulos})
$$
0=(\nabla^B_uT^B)(x,y,z)=\sigma_{x,y,z}g(T^B(x,y),T^B(z,u)),
$$
then by \cite[Formula (3.18)]{ivanov-papadopoulos}
$$
(\nabla^{LC}_xT^B)(y,z,u)=(\nabla^B_xT^B)(y,z,u)+\frac{1}{2}
\sigma_{x,y,z}g(T^B(x,y),T^B(z,u))=0\,.
$$
\end{proof}
Notice that if the dimension of $X$ is $2$, then (cf. \cite[(2.14)]{ivanov-papadopoulos})
$$
T^B=-*\theta
$$
hence, on a complex surface
$$
\nabla^{LC}T^B=0\quad\iff\quad \nabla^{LC}\theta=0\,.
$$
Then as a consequence of Theorem \ref{thm:bianchi-skt-then-levicivita-parallel} one has
\begin{cor}
Let $X$ be a complex surface and $g$ be a Gauduchon metric such that the Bismut connection satisfies the first Bianchi identity. Then,
$g$ is Vaisman.
\end{cor}
In fact, we  can prove that also its converse is true and show  Theorem A.

\begin{proofof}{Theorem A}
Since Vaisman metrics are Gauduchon (or equivalently SKT in complex dimension $2$), we just need to prove that if $g$ is Vaisman then the Bismut connection satisfies the first Bianchi identity.
Recall that on complex surfaces by \cite[Appendix A]{agricola-ferreira}
$$
\nabla^BT^B=\nabla^{LC}T^B.
$$
Since, by hypothesis $\nabla^{LC}\theta=0$ then 
$\nabla^BT^B=\nabla^{LC}T^B=0$, hence by Theorem \ref{thm:parallel-bianchi-iff-skt}
the Bismut connection satisfies the first Bianchi identity.
\end{proofof}

Notice that, up to now we have not used the second part of the definition of Bismut K\"ahler-like metrics. First of all, notice that if
$$
R^B(x,y,z,w)=R^B(Jx,Jy,z,w)
$$
then for the Ricci form of the Bismut connection we have
$$
\rho^B(x,y)=\frac{1}{2}\sum R^B(x,y,e_i,Je_i)=
\frac{1}{2}\sum R^B(Jx,Jy,e_i,Je_i)=\rho^B(Jx,Jy)
$$
where $\left\lbrace e_i\right\rbrace$ denotes an orthonormal basis of the tangent space. Namely, $\rho^B$ is a $(1,1)$-form.
In particular,
$$
d(\rho^B)^{1,1}=d\rho^B=0
$$
where $(\beta)^{1,1}$ denotes the $(1,1)$ component of a $2$-form $\beta$.
Hence, since
$$
(\rho^B)^{1,1}=\rho^{Ch}+\frac{dJ\theta+JdJ\theta}{2}
$$
where $\rho^{Ch}$ is the Ricci form of the Chern connection, one has that
$$
dd^c\theta=0.
$$
As a consequence,
\begin{prop}
Let $X$ be a complex surface and let $g$ be a Hermitian metric such that
$$
R^B(x,y,z,w)=R^B(Jx,Jy,z,w)
$$
for every tangent vectors $x,y,z,w$.
Then, $g$ is locally conformally K\"ahler.
\end{prop}
\begin{proof}
Let $\omega$ be the fundamental form associated to $g$. Then, $d\omega=\theta\wedge\omega$ and so we need to prove that $d\theta=0$.\\
First of all, notice that in any dimension $d\theta$ is a primitive $2$-form, indeed differentiating $d\omega^{n-1}=\theta\wedge\omega^{n-1}$ one has
$$
0=d\theta\wedge\omega^{n-1}=L^{n-1}d\theta\,.
$$
where $L=\omega\wedge-$ is the operator of multiplication by $\omega$ acting on forms.
Therefore, we compute
$$
d^*d\theta=-*d(*d\theta)=*JJ^{-1}dJd\theta=*Jdd^c\theta,
$$
where in the second equality we have used for instance \cite[Formula (A8)]{vaisman}.
Then,
$$
dd^c\theta=0\quad\iff\quad d^*d\theta=0 \quad\iff\quad d\theta=0.
$$
\end{proof}

\subsection{SKT metrics with Lee potential}\label{subs:LP}

In \cite{belgun-threefolds} a generalization of lck metrics with potential (and so also of Vaisman metrics) is introduced and it is shown that these metrics exist on Calabi-Eckmann manifolds.

\begin{definition} $($\cite{belgun-threefolds,yau-zhao-zheng}$)$  A Hermitian manifold  $(X, J, g)$  is called  \emph{LP} (or equivalently it has  \emph{Lee potential})  if  the $(1,0)$-part of the Lee form $\theta$ of $g$, namely $\eta:= \theta^{1,0}$, satisfies
$$
\eta\neq 0\,, \quad \del\eta=0\,, \quad \del\omega=c\,\eta\wedge\del\bar\eta\,,
$$ 
for some non-zero constant $c$.\\
If in addition $\nabla^BT^B=0$ the metric $g$ is called \emph{Generalized Calabi-Eckmann} (\emph{GCE} for short).
\end{definition}

As a consequence of Theorem \ref{thm:parallel-bianchi-iff-skt} the Bismut connection of a Generalized Calabi-Eckmann SKT metric satisfies the first Bianchi identity.

We show that if a SKT metric is LP then all its powers are $\del\delbar$-closed.
\begin{prop}
Let $X$ be a complex manifold of complex dimension $n$ endowed with a SKT metric $\omega$ with Lee potential. Then,
$$
\del\delbar\omega^k=0, \qquad \forall 1\leq k\leq n-1\,.
$$
In particular, $\omega$ is $k^{\text{th}}$-Gauduchon for every $1\leq k\leq n-1$ i.e., 
$$
\del\delbar(\omega^k)\wedge\omega^{n-k-1}=0,
$$
and astheno-K\"ahler, i.e., $\del\delbar\omega^{n-2}=0$.
\end{prop}
\begin{proof}
First of all, notice that for any $k$
$$
\del\delbar\omega^k=k(\del\delbar\omega\wedge\omega-
(k-1)\delbar\omega\wedge\del\omega)\wedge\omega^{k-2}
$$
Since, $\omega$ is SKT and LP one has that
$$
\del\delbar\omega=0\,,\qquad
\delbar\omega\wedge\del\omega=
|c|^2\bar\eta\wedge\delbar\eta\wedge\eta\wedge\del\bar\eta\,,
$$
hence
$$
\del\delbar\omega^k=-k(k-1)|c|^2(\bar\eta\wedge\delbar\eta\wedge\eta\wedge\del\bar\eta)
\wedge\omega^{k-2}\,.
$$
Moreover, by the SKT and LP conditions
one has also
$$
\del\eta=0,\quad
0=\del\delbar\omega=\bar c\del\bar\eta\wedge\delbar\eta
$$
concluding the proof.
\end{proof}
\begin{rem}
Observe that in \cite[Remark 2]{yau-zhao-zheng} it was proved that a Hermitian metric that is both SKT and Gauduchon then it is $k^{\text{th}}$-Gauduchon.
\end{rem}

The LP assumption is fundamental in the previous Proposition, indeed in the following example we exhibit a manifold with a SKT metric $\omega$ which is not LP and $\del\delbar\omega^{n-2}\neq 0$, namely it is not astheno-K\"ahler.

\begin{example}
Let $(G,J)$ be the $4$-dimensional nilpotent Lie group equipped with a left invariant complex structure $J$ with structure equations
$$
\left\lbrace
\begin{array}{lcl}
d\,\varphi^1 & =& 0\\
d\,\varphi^2 & =& 0\\
d\,\varphi^3 & =& \lambda_1\varphi^{1\bar 1}+i\,a\varphi^{2\bar 2}\\
d\,\varphi^4 & =& \lambda_2\varphi^{2\bar 2}\,.
\end{array}
\right.
$$
with respect to a left invariant unitary coframe
$\left\lbrace\varphi^i\right\rbrace_{i=1,\cdots,4}$, where $\lambda_1,\lambda_2, a$ are real numbers and $\lambda_1,\lambda_2>0$.
It is easy to see that left-invariant Hermitian metric $
\omega:=\frac{i}{2}\sum_{j=1}^4\varphi^j\wedge\bar\varphi^j$ is SKT.
Now we show that it is not LP.\\
By explicit computation one can show that the Gauduchon torsion $(1,0)$-form
$\eta$ is
$$
\eta=\lambda_2\varphi^4+(\lambda_1-ia)\varphi^3.
$$
In particular, computing
$$
\eta\wedge\del\bar\eta=-\lambda_2(-\lambda_2^2+ia\lambda_1-a^2)\varphi^{24\bar 2}+\lambda_1(\lambda_1+ia)\varphi^{14\bar1}-
(\lambda_1-ia)(-\lambda_2^2+ia\lambda_1-a^2)\varphi^{23\bar2}+
$$
$$
+
\lambda_1(\lambda_1-ia)(\lambda_1+ia)\varphi^{13\bar 1}
$$
and 
$$
\del\omega=-\lambda_1\varphi^{13\bar 1}+ia\varphi^{23\bar2}-\lambda_2\varphi^{24\bar2}
$$
One gets that $\del\omega=c\eta\wedge\del\bar\eta$ for some non-zero constant $c$ if and only if $\lambda_1=-ia=0$ but $\lambda_1>0$ so the metric $\omega$ is not LP.\\
Moreover, notice that the metric $\omega$ is not astheno-K\"ahler, indeed one has
$$
\del\delbar\omega^2=2\lambda_1\lambda_2\varphi^{123\bar1\bar2\bar4}+
2\lambda_1\lambda_2\varphi^{124\bar1\bar2\bar3}\neq 0.
$$
\end{example}

\subsection{Behavior under small deformations}

In this section we partially answer to a question proposed in \cite{angella-otal-ugarte-villacampa} about the stability of the Bismut K\"ahler-like property under small deformations of the complex structure. In particular, we show the following
\begin{prop}\label{prop:deformations}
The existence of a SKT metric with Bismut connection satisfying the first Bianchi identity and with $\nabla^BT^B=0$ on compact complex manifolds is not an open property under small deformations of the complex structure.  
\end{prop}
\begin{proof}
Let $X=\mathbb{S}^3\times\mathbb{S}^3$ and $\mathbb{S}^3\simeq\SU(2)$ be the Lie group of special unitary $2\times 2$ matrices and denote by $\mathfrak{su}(2)$ its Lie algebra. Denote by 
$\{e_1, e_2, e_3\}$, $\{f_1, f_2, f_3\}$  a basis of the first copy of $\mathfrak{su}(2)$, respectively of the second copy of $\mathfrak{su}(2)$ and by  
$\{e^1, e^2, e^3\}$, $\{f^1, f^2, f^3\}$ the corresponding dual co-frames. Then we have the following commutation relations:
$$
[e_1,e_2]=2e_3\,,\quad [e_1,e_3]=-2e_2,\quad [e_2,e_3]=2e_1\,,
$$
and the corresponding Cartan structure equations
\begin{equation}
\left\{\begin{array}{rcl}
            de^1 &=&          -  2\, e^2 \wedge e^3 \\[5pt]
            de^2 &=& \phantom{+} 2\, e^1 \wedge e^3 \\[5pt]
            de^3 &=&          -  2\, e^1 \wedge e^2 \\[5pt]
            df^1 &=&          -  2\, f^2 \wedge f^3 \\[5pt]
            df^2 &=& \phantom{+} 2\, f^1 \wedge f^3 \\[5pt]
            df^3 &=&          -  2\, f^1 \wedge f^2
           \end{array}\right. \;.
\end{equation}
Define a complex structure $J$ on $X$ by setting 
$$
Je_1=e_2,\,\quad Jf_1=f_2,\,\quad Je_3=f_3\,.
$$
Therefore a complex co-frame of $(1,0)$-forms for $J$ is given by 
\begin{equation}\label{complexcoframe}
\left\{\begin{array}{rcl}
            \varphi^1 &:=& e^1 + i\, e^2 \\[5pt]
            \varphi^2 &:=& f^1 + i\, f^2 \\[5pt]
            \varphi^3 &:=& e^3 + i\, f^3
\end{array}\right. \;.
\end{equation}
In particular the complex structure equations are given by
\begin{equation}\label{complexstructureequations}
\left\{\begin{array}{rcl}
            d\varphi^1 &=&     i\varphi^1\wedge\varphi^3+
            					i\varphi^1\wedge\overline{\varphi}^3 \\[5pt]
            d\varphi^2 &=&     \varphi^2\wedge\varphi^3-
            					\varphi^2\wedge\overline{\varphi}^3 \\[5pt]
            d\varphi^3 &=&     -i\varphi^1\wedge\overline{\varphi}^1+
            					\varphi^2\wedge\overline{\varphi}^2
       \end{array}\right. \;,
 \end{equation}
Note that $(X,J)$ is a central Calabi-Eckmann threefold and in \cite{wang-yang-zheng} it is showed that this complex manifold admits a Bismut-flat Hermitian metric, which in particular is SKT and satisfies the first Bianchi identity.\\
Now let $J_t$ be the almost complex structure on $X$ considered in \cite{tardini-tomassini}
defined as
\[
\left\{\begin{array}{rcl}
            \varphi^1_t &:=& \varphi^1 \\[5pt]
            \varphi^2_t &:=& \varphi^2 \\[5pt]
            \varphi^3_t &:=& \varphi^3-t\overline{\varphi}^3
       \end{array}
\right. ;
\]
then, using the structure equations \eqref{complexstructureequations}, a straightforward computation yields to
\[
\left\{\begin{array}{rcl}
            d\varphi^1_t &=& \frac{i(\bar{t}+1)}{1-\mid t\mid^2}
            \varphi ^{13}_t + \frac{i(t+1)}{1-\mid t\mid^2}
            \varphi ^{1\bar{3}}_t \\[10pt]
            d\varphi^2_t &=& \frac{1-\bar{t}}{1-\mid t\mid^2}
            \varphi ^{23}_t + \frac{t-1}{1-\mid t\mid^2}
            \varphi ^{2\bar{3}}_t \\[10pt]
            d\varphi^3_t &=& i(t-1)
            \varphi ^{1\bar{1}}_t + (t+1)
            \varphi ^{2\bar{2}}_t
           \end{array}\right. \;
\]
and consequently $J_t$ is integrable. Set $X_t=(X,J_t)$, in \cite[Remark 3.6]{tardini-tomassini} it was proven that when $|t|^2+\text{Re}\,t-\text{Im}\,t\neq 0$ then $X_t$ does not admit any SKT metric. Therefore by Corollary \ref{cor:bianchi-parallel-iff-skt}, for such values of $t$, $X_t$ does not admit any Hermitian metric whose Bismut connection satisfies the first Bianchi identity and with parallel torsion.
\end{proof}

\section{Bismut K\"ahler-like almost abelian Lie groups}\label{sect:almost-abelian}

In this Section, we study the existence of Hermitian metrics with K\"ahler-like Bismut connection on simply-connected, almost abelian Lie groups in order to give new examples besides the ones provided in \cite{angella-otal-ugarte-villacampa} on $6$-dimensional solvmanifolds with invariant complex structures, trivial canonical bundle and  invariant Hermitian metrics.
Let $G$ be a simply-connected, almost abelian Lie group
namely, its Lie algebra $\mathfrak{g}$ has a codimension-one abelian ideal $\mathfrak{n}$. In particular, notice that such a $G$ is solvable.
Let $(J,g)$ be a left-invariant Hermitian structure on $G$, therefore there exists a basis $\left\lbrace e_1,\,\cdots\,,e_{2n}\right\rbrace$ on $\mathfrak{g}$ such that, setting
$\mathfrak{n}_1=\text{Span}_{\mathbb{R}}\left\langle e_2,\,\cdots\,,e_{2n-1}\right\rangle$, one has
$$
\mathfrak{n}=\text{Span}_{\mathbb{R}}\left\langle e_1,\,\cdots\,,e_{2n-1}\right\rangle\,,
\quad
Je_1=e_{2n},\quad
J(\mathfrak{n}_1)\subset \mathfrak{n}_1\,.
$$
By \cite{lauret-valencia} the complex structure $J$ is integrable if and only if ad $e_{2n}$ leaves $\mathfrak{n}_1$
invariant, and $A := (\text{ad}\,e_{2n} )|_{\mathfrak{n}_1}$ commutes with $J_1:=J|_{\mathfrak{n}_1}$ . Hence one has
$$
\text{ad}\,e_{2n}= 
\left(\begin{array}{rcl}
           a & 0 & 0\\
           v & A & 0\\
           0 & 0 & 0
\end{array}\right) 
$$
with $a\in\mathbb{R}$, $v\in\mathfrak{n}_1$, $A\in\mathfrak{gl}(\mathfrak{n}_1)$ and $[A,J_1]=0$.\\
In particular, we have the following non-trivial Lie brackets
$$
[e_{2n},x]=Ax\,, \quad
[e_{2n},e_1]=ae_1+v
$$
for any $x\in \mathfrak{n}_1$.\\
We fix a real inner product $g$ on $\mathfrak{g}$ with an orthogonal decomposition
$$
\mathfrak{g}=\mathbb{R}e_1\oplus\mathfrak{n}_1\oplus\mathbb{R}e_{2n}
$$
and
$$
\mathfrak{n}=\mathbb{R}e_1\oplus\mathfrak{n}_1
$$
and a compatible integrable complex structure $J$ such that $Je_1=e_{2n}$ and
$J(\mathfrak{n}_1)\subset \mathfrak{n}_1$.\\
We recall that by, \cite[Lemma 4.2]{arroyo-lafuente} the metric $g$ is SKT if and only if
$$
aA+A^2+A^tA\in\mathfrak{so}(\mathfrak{n_1})\,.\\
$$
In order to study the existence of Bismut K\"ahler-like metrics we need the expression of the Bismut connection. By using the formula (see \cite{dotti-fino})
$$
2g\left(\nabla^B_xy,z\right) =  g\left([x,y]-[Jx,Jy],z\right)-
g\left([x,z]-[Jx,Jz],y\right)-
g\left([y,z]+[Jy,Jz],x\right),
$$
we get the following

\begin{lemma}
Let $G$ be an almost abelian Lie group. Then, with the previous notations, the only non-zero components of the Bismut connection are
$$
\begin{aligned}
\nabla^B_{e_1}e_1 & =   ae_{2n}\,,\\
\nabla^B_{x}e_1 & =  \sum_{j=2}^{2n-1}g(S(A)Jx\,,\,e_j)e_j+g(v,x)e_{2n}\,,\\
\nabla^B_{e_1}e_{2n}&=-ae_{1}\,,\\
\nabla^B_{x}y &=g(x,S(A)Jy)e_1+g(x,S(A)y)e_{2n}\,,\\
\nabla^B_{e_1}y &=-\sum_{j=2}^{2n-1}g(S(A)Jy,e_j)e_j\,,\\
\nabla^B_{e_{2n}}y & =\frac{1}{2}\sum_{j=2}^{2n-1}g((A-A^t)y,e_j)e_j\,,\\
\nabla^B_{x}e_{2n} & =-g(v,x)e_1-\sum_{j=2}^{2n-1}g(S(A)x,e_j)e_j\,,
\end{aligned}
$$
with $x,y\in\mathfrak{n}_1$ and $S(A):=\frac{1}{2}(A+A^t)$.\\
\end{lemma}

In particular if we assume that $A\in\mathfrak{so}(\mathfrak{n}_1)$ the non-trivial components reduce to
$$
\begin{aligned}
\nabla^B_{e_1}e_1&=ae_{2n}\,,\\
\nabla^B_{x}e_1&=g(v,x)e_{2n}\,,\\
\nabla^B_{e_1}e_{2n}&=-ae_{1}\,,\\
\nabla^B_{e_{2n}}y&=Ay\,,\\
\nabla^B_{x}e_{2n}&=-g(v,x)e_1\,.
\end{aligned}
$$
If $y\in\mathfrak{n}_1$, by explicit computations one gets that that the only non-trivial components of the Bismut curvature tensor are
$$
\begin{aligned}
R^B(e_{2n},y)e_1&=-g(v,Ay)e_{2n}\,,\\
R^B(e_{2n},y)e_{2n}&=g(v,Ay)e_{1}\,,\\
R^B(e_{1},e_{2n})e_1&=(a^2+|v|^2)e_{2n}\,,\\
R^B(e_{1},e_{2n})e_{2n}&=(-a^2-|v|^2)e_{1}
\end{aligned}
$$
together with their symmetries.
We can now prove
\begin{theorem}
Let $G$ be an almost abelian Lie group and assume that $A\in\mathfrak{so}(\mathfrak{n}_1)$. Then, the Bismut connection is K\"ahler-like if and only if $Av$ is $g$-orthogonal to $\mathfrak{n}_1$.
\end{theorem}
\begin{proof}
One can check that the first Bianchi identity holds if and only if $g(v,Ay)=0$ for any $y\in\mathfrak{n}_1$. Indeed, for instance
$$
R^B(y,e_{2n})e_1+R^B(e_{2n},e_1)y+R^B(e_1,y)e_{2n}=
g(v,Ay)e_{2n}
$$
and similarly for the other relations.
For the $J$-invariance in the first two components of $R^B(\cdot,\cdot,\cdot,\cdot)$ one has the same conclusion, for example
$$
R^B(e_{2n},y,e_1,e_{2n})-R^B(Je_{2n},Jy,e_1,e_{2n})=-g(v,Ay)
$$
and so on.
Hence, the Bismut connection is K\"ahler-like if and only if 
$g(v,Ay)=0$ for every $y\in\mathfrak{n}_1$ if and only if $Av$ is $g$-orthogonal to $\mathfrak{n}_1$.
\end{proof}
\begin{rem}
Notice that in general these Bismut K\"ahler-like metrics are not K\"ahler. 
More precisely, the torsion of the Bismut connection is given by
$$
\begin{aligned}
T^B(e_1,y)&=-g(v,y)e_{2n}\,,\\
T^B(e_{2n},y)&=g(v,y)e_{1}\,,\\
T^B(e_1,e_{2n})&=v
\end{aligned}
$$
for $y\in\mathfrak{n}_1$.\\
And so the only non-zero components of the torsion $3$-form are
$$
T^B(e_1,y,e_{2n})=-g(v,y)\,.
$$
In particular, one can check explicitly that $\nabla^BT^B=0$. 
Hence, these exist  explicit examples of metrics that are SKT, with $\nabla^BT^B=0$ but they do not satisfy the condition $R^D(x,y,z,w) = R^D(z,y,x,w),$ for every tangent vectors $x,y,z,w$.

\end{rem}

\begin{rem}
We notice that having $A\in\mathfrak{so}(\mathfrak{n}_1)$ is not necessary in order to have a Hermitian metric with Bismut K\"ahler-like connection on an almost abelian simply-connected Lie group. Indeed, by \cite{angella-otal-ugarte-villacampa} the Lie algebra $\mathfrak{h}_8$ with structure equations $(0,0,0,0,12)$ is an almost abelian Lie algebra with Bismut K\"ahler-like connection but the corresponding matrix $A$ is not antisymmetric.
\end{rem}

We now discuss the existence of a compact quotient on an explicit example in dimension $6$.
First of all notice that if $\mathfrak{g}$ is an almost abelian Lie algebra then it is unimodular (namely all the adjoint maps are traceless) if and only if $a+\text{tr}\,A=0$.\\
If $A \in\mathfrak{so}(\mathfrak{n}_1)$, then $\mathfrak{g}$ is unimodular if and only if $a=0$.
\begin{example}
Let $\mathfrak{g}$ be an almost abelian Lie algebra of dimension $6$
with non trivial brackets
$$
[e_6,e_1]=e_2\,,\quad [e_6,e_4]=e_5\,,\quad [e_6,e_5]=-e_4\,.
$$
Hence we have
$$
\tilde A:=\text{ad}\,e_{6}= 
\left(\begin{array}{rcl}
           a & 0 & 0\\
           v & A & 0\\
           0 & 0 & 0
\end{array}\right) 
$$
with $v=e_2$, $a=0$ and
$$
A=\left(\begin{array}{rclc}
           0 & 0 & 0 & 0\\
           0 & 0 & 0 & 0\\
           0 & 0 & 0 & -1\\
           0 & 0 & 1 & 0
\end{array}\right).
$$
We set $\varphi(t)=e^{t\tilde A}$, by \cite{bock} $\mathfrak{g}$ admits a lattice if and only if there exists $t_0\in\mathbb{R}$ such that
$\varphi(t_0)$ is conjugate to an integral matrix. Set $t_0:=\pi$, then
$$
\varphi(t_0)= \left(\begin{array}{rclclc}
           1 & 0 & 0 & 0 & 0 & 0\\
           \pi & 1 & 0 & 0 & 0 & 0\\
           0 & 0 & 1 & 0 & 0 & 0\\
           0 & 0 & 0 & -1 & 0  & 0\\
           0 & 0 & 0 & 0 & -1 & 0\\
           0 & 0 & 0 & 0 & 0 & 1
\end{array}\right).
$$
This matrix is conjugate to
$$
\left(\begin{array}{rclclc}
           2 & -1 & 1 & 0 & 0 & 0\\
           1 & 0 & 1 & 0 & 0 & 0\\
           0 & 0 & 1 & 0 & 0 & 0\\
           0 & 0 & 0 & -1 & 0  & 0\\
           0 & 0 & 0 & 0 & -1 & 0\\
           0 & 0 & 0 & 0 & 0 & 1
\end{array}\right).
$$
Hence, there exists a lattice $\Gamma$ in $G$ such that $\Gamma\backslash G$ is an almost-albelian solvmanifold of dimension $6$ admitting a Bismut K\"ahler-like Hermitian metric.
\end{example}

In relation  to the condition
\begin{equation}\label{symmetry13}
G^B(x,y,z,w):=R^B(x,y,z,w)-R^B(z,y,x,w)=0,
\end{equation} 
for every tangent vector $x,y,z,w$.  we can  prove the following
\begin{theorem}\label{thm:almost-abelian-skl-zz}
An almost abelian Lie group $G$ admits a left-invariant  SKT  metric satisfying \eqref{symmetry13}  if and only if
$$
a=0,\quad v=0,\quad A\in\mathfrak{so}(\mathfrak{n}_1).
$$
\end{theorem}
\begin{proof}
We have
$$
G^B(e_1,e_1,e_{2n},e_{2n})=-R^B(e_{2n},e_1,e_1,e_{2n})=-g(R^B(e_{2n},e_1)e_1,e_{2n})=
$$
$$
=-g\left(\nabla^B_{e_{2n}}\nabla^B_{e_1}e_1-
\nabla^B_{e_1}\nabla^B_{e_{2n}}e_1-
\nabla^B_{[e_{2n},e_1]}e_1,e_{2n}\right)\,.
$$
Since
$$
\nabla^B_{e_1}e_1=ae_{2n},\quad \nabla^B_{e_{2n}}e_1=0,\quad
\nabla^B_{e_{2n}}e_{2n}=0
$$
and for every $x\in\mathfrak{n}_1$
$$
\nabla^B_xe_1=\sum_{j=2}^{2n-1}g(S(A)Jx,e_j)e_j+g(v,x)e_{2n}\,,
$$
where $S(A):=\frac{1}{2}(A+A^t)$,
we have,
$$
G^B(e_1,e_1,e_{2n},e_{2n})=a^2+\sum_{j=1}^{2n-2}v_i^2,
$$
hence $G^B(e_1,e_1,e_{2n},e_{2n})=0$ if and only if $a=0$ and
$v=0$.
Similarly, assuming that $a=0$ and $v=0$ we get, for every $x\in\mathfrak{n}_1$
$$
G^B(e_1,e_1,x,x)=\frac{1}{2}\sum_{j=2}^{2n-1}\left[g(S(A)Jx,e_j)\right]^2
$$
hence $G^B(e_1,e_1,x,x)=0$ for every $x\in\mathfrak{n}_1$ if and only if
$g(S(A)Jx,e_j)=0$ for every $x\in\mathfrak{n}_1$ and every $j=1,\cdots,2n$ (notice that for $j=1$ and $j=2n$ it is obvious) if and only if $S(A)=0$.\\
Thus, if $g$  satisfies \eqref{symmetry13} then we have just proven that $a=0$, $v=0$, and $A\in\mathfrak{so}(\mathfrak{n}_1)$. Now we show the viceversa, suppose that $a=0$, $v=0$ and $A\in\mathfrak{so}(\mathfrak{n}_1)$ then
by similar computations
one can show that the only non trivial components of the Bismut connection
are, for every $y\in\mathfrak{n}_1$
$$
\nabla^B_{e_{2n}}y=\sum_{j=2}^{2n-1}g(Ay,e_j)e_j
$$
hence, by definition
$$
R^B(x,y)z=\nabla^B_{x}\nabla^B_{y}z-
\nabla^B_{y}\nabla^B_{x}z-
\nabla^B_{[x,y]}z=0
$$
for every $x,y,z\in\mathfrak{g}$, hence $R^B(x,y,z,w)=0$ for every $x,y,z,w\in\mathfrak{g}$ concluding the proof.
\end{proof}
Notice that, in particular, we have proven that an almost abelian Lie group $G$ has
$R^B(x,y)z=0$ for every $x,y,z\in\mathfrak{g}$
if and only if
$$
a=0,\quad v=0,\quad A\in\mathfrak{so}(\mathfrak{n}_1),
$$
in particular almost abelian  Lie groups  satisfying \eqref{symmetry13} are Bismut flat. In fact, we prove the following

\begin{theorem} 
An almost abelian Lie group $G$ with
$$
v=0\quad\text{and}\quad A\in\mathfrak{so}(\mathfrak{n}_1)
$$
is K\"ahler. In particular, the only almost abelian examples   satisfying \eqref{symmetry13} are K\"ahler and flat.
\end{theorem}
\begin{proof}
If $v=0$ and $A\in\mathfrak{so}(\mathfrak{n}_1)$
then the only non-zero commutators become
$$
[e_{2n},x]=Ax,\quad [e_{2n},e_1]=ae_1
$$ 
where $x\in\mathfrak{n}_1$.
Hence, computing explicitly the Bismut connection for $(g,J)$ as before one gets that  the only non-trivial components are
$$
\nabla^B_{e_1}e_1=ae_{2n},\quad \nabla^B_{e_{1}}e_{2n}=-ae_1,\quad
\nabla^B_{e_{2n}}y=\sum_{j=2}^{2n-1}g(Ay,e_j)e_j
$$
where $y\in\mathfrak{n}_1$.\\
Computing explicitly the torsion
$$
T^B(x,y)=\nabla_xy-\nabla_yx-[x,y]
$$
one gets $T^B=0$. For instance we compute here, for
$y\in\mathfrak{n}_1$,
$$
T^B(e_{2n},y)=\sum_{j=2}^{2n-1}g(Ay,e_j)e_j-[e_{2n},y]=
\sum_{j=2}^{2n-1}g(Ay,e_j)e_j-Ay=0\,.
$$
\end{proof}

\section{Bismut-K\"ahler like condition and pluriclosed flow}

Let $(X,J)$ be a complex manifold of complex dimension $n$.
In \cite{streets-tian} the authors introduce a family of flows called \emph{Hermitian curvature flows}, among these a particular one is the so called \emph{pluriclosed flow} which is defined by the equation
$$
\frac{\del}{\del\,t}\omega(t)=-\left(\rho^B(\omega)\right)^{1,1}\,,\qquad
\omega(0)=\omega_0
$$
where $\left(\rho^B(\omega)\right)^{1,1}$ denotes the $(1,1)$-part of the Ricci form of the Bismut connection and $\omega_0$ is a fixed Hermitian metric. It is easy to see that this flow preserves the SKT condition.\\
In this section we study the behavior of the Bismut K\"ahler-like condition under the pluriclosed flow on complex surfaces and on almost abelian Lie groups considered in the previous section.

\subsection{Pluriclosed flow on Vaisman surfaces}

We recall that a Vaisman metric on a complex manifold $(X\,,J)$ is a Hermitian metric $\omega$ such that
$$
d\omega=\theta\wedge\omega\quad\text{and}\quad\nabla^{LC}\theta=0,
$$
where $\theta$ is a $1$-form.\\
In fact, a Vaisman structure on a complex manifold is uniquely determined (up to a positive constant) by its Lee form $\theta$ via the following
$$
\omega=\frac{1}{|\theta|^2}(\theta\wedge J\theta-dJ\theta).
$$
On complex dimension $2$, Belgun in \cite{belgun-surfaces} classified those compact complex surfaces admitting a Vaisman metric, and they are properly elliptic surfaces, primary and secondary Kodaira surfaces, elliptic Hopf surfaces and Hopf surfaces of class $1$.\\
Notice that in particular, on a compact complex surface, Vaisman metrics are SKT. We ask whether the Vaisman condition is preserved along the pluriclosed flow. We answer this question in the case the initial metric has constant scalar curvature.

In order to prove Theorem B we first need the following Lemma.

\begin{lemma}\label{lemma:chern-ricci-vaisman}
Let $X$ be a compact complex surface and let $\omega$ be a Vaisman metric on $X$ with Lee form $\theta$. Then, the Ricci form of the Chern connection is
$$
\rho^{Ch}=h\,dJ\theta
$$
for some $h\in\mathcal{C}^\infty(X,\mathbb{R})$.\\
Moreover, the scalar curvature of $\omega$ is constant if and only if $h$ is constant and, in particular, in such a case $c_1(X)=0$.
\end{lemma}

\begin{proofof}{Theorem B}
Let $(\omega_0,\theta_0)$ be a Vaisman structure on $X$ with constant scalar curvature; then, by Lemma \ref{lemma:chern-ricci-vaisman}, we have
$$
c_1(X)=0\,.
$$
Moreover, since $\omega_0$ is Vaisman, then
$$
\omega_0=\frac{1}{|\theta_0|_0^2}(\theta_0\wedge J\theta_0-dJ\theta_0)
$$
where with $|\cdot|_0$ we denote the norm with respect to $\omega_0$ and in particular $|\theta_0|_0$ is constant.\\
We claim that the metrics
$$
\omega_t=\frac{1}{|\theta_0|_0^2}(\theta_0\wedge J\theta_0-f(t)dJ\theta_0)\,,
$$
for a suitable positive function $f(t)$ depending only on $t$ with $f(0)=1$, are Vaisman and they are solutions of the pluriclosed flow.\\
First of all notice that the corresponding Lee form of $\omega_t$ is
$$
\theta_t=\frac{1}{f(t)}\theta_0
$$
indeed, $d\omega_t=\theta_t\wedge\omega_t$ and clearly $\theta_t$ is closed, hence they are locally conformally K\"ahler.\\
Moreover, the Lee field $\theta^{\sharp_t}_t$ with respect to $\omega_t$ is
$$
\theta^{\sharp_t}_t=\frac{1}{f(t)}\theta^{\sharp_0}_0
$$
hence by \cite[Theorem 1]{moroianu-moroianu-ornea} the metrics $\omega_t$ are Vaisman.\\
In fact, 
recall that on a compact complex manifold admitting Vaisman metrics, the Lee vector fields of all Vaisman structures are holomorphic, and coincide up to a positive multiplicative constant (cf. \cite{tsukada} and \cite{madani-moroianu-pilca}).\\
Moreover, by a straightforward computation,
$$
\omega_t^2=f(t)^2\omega^2_0
$$
therefore, the Ricci forms of the Chern connection of $\omega_t$ and
$\omega_0$ must coincide
$$
\rho^{Ch}_{\omega_t}=\rho^{Ch}_{\omega_0}\,.
$$
Hence,  the Ricci forms of the Bismut connection of $\omega_t$ and
$\omega_0$ are related by
$$
\rho^{B}_{\omega_t}=\rho^{Ch}_{\omega_t}-dJ\theta_t=
\rho^{Ch}_{\omega_0}-\frac{1}{f(t)}dJ\theta_0\,.
$$
Notice that, since $\omega_t$ is lck then $\rho^{B}_{\omega_t}$ is a closed $(1,1)$-form.\\
For the pluriclosed flow we have
$$
\frac{\partial}{\partial\,t}\omega_t=
-\frac{1}{|\theta_0|_0^2}f'(t)dJ\theta_0=d(-\frac{1}{|\theta_0|_0^2}f'(t)J\theta_0)
$$
hence it is necessary to have $c_1(X)=0$ to have $\omega_t$ solving 
$$
\frac{\partial}{\partial\,t}\omega_t=-(\rho^{B}_{\omega_t})^{1,1}
$$
because we have
$$
d(-\frac{1}{|\theta_0|_0^2}f'(t)J\theta_0)=
\rho^{Ch}_{\omega_0}-d(\frac{1}{f(t)}J\theta_0)\,
$$
so $\rho^{Ch}_{\omega_0}$ is an exact form and, up to a constant, it represents the first Chern class of $X$.\\
Now, by Lemma \ref{lemma:chern-ricci-vaisman}
$$
\rho^{Ch}_0=h\,dJ\theta_0
$$
for some constant $h$.
Hence,
$$
\rho^{B}_{\omega_t}=
\left(h-\frac{1}{f(t)}\right)dJ\theta_0\,.
$$
and so
$$
-(\rho^{B}_{\omega_t})^{1,1}=-\rho^{B}_{\omega_t}=
-\left(h-\frac{1}{f(t)}\right)dJ\theta_0\,.
$$
The equation of the pluriclosed flow reduces to find a solution $f(t)$ of
$$
\frac{1}{|\theta_0|_0^2}f'(t)dJ\theta_0=
\left(h-\frac{1}{f(t)}\right)dJ\theta_0
$$
or equivalently
$$
\left(\frac{1}{|\theta_0|_0^2}f'(t)-
h+\frac{1}{f(t)}\right)dJ\theta_0=0\,.
$$
In fact, the equation
$$
\frac{1}{|\theta_0|_0^2}f'(t)-h+\frac{1}{f(t)}=0
$$
admits a unique solution $f(t)>0$ and with $f(0)=1$. Therefore,
$(\omega_t,\theta_t)$ are Vaisman metrics and solutions of the pluriclosed flow.
\end{proofof}

\begin{rem}
We notice that, if $X$ is a compact complex surface admitting a Vaisman metric $\omega_0$ with constant scalar curvature, then the pluriclosed flow starting with $\omega_0$ preserves such a condition. Indeed,
by Lemma \ref{lemma:chern-ricci-vaisman} one has for every $t\neq 0$
$$
\rho^{Ch}_{\omega_t}=h_tdJ\theta_t
$$
for some smooth function $h_t$ on $X$, and for $t=0$
$$
\rho^{Ch}_{\omega_0}=h_0dJ\theta_0
$$
with $h_0$ constant by hypothesis.\\
Then, by the proof of Theorem B we have that
$$
\rho^{Ch}_{\omega_t}=\rho^{Ch}_{\omega_0}\quad\text{and}\quad
\theta_t=\frac{1}{f(t)}\theta_0
$$
therefore $h_t=f(t)h_0$ for every $t$, giving $h_t$ constant and so $\omega_t$ has constant scalar curvature.
\end{rem}

We now give a proof of Lemma \ref{lemma:chern-ricci-vaisman}

\begin{proofof}{Lemma \ref{lemma:chern-ricci-vaisman}}
Let $\omega$ be a Vasiman metric on $X$. Then, by
\cite[Lemma 4.4]{alexandrov-ivanov} the Ricci forms of the Bismut and Weyl connections coincide
$$
\rho^B=\rho^W\,,
$$
where the Weyl connection determined by the Hermitian structure $g$ of $X$ is the unique torsion-free connection $\nabla^W$ such that $\nabla^Wg=\theta\otimes g$. In fact, the Weyl connection is related with the Levi-Civita connection of $g$ via
$$
\nabla^W_xy=\nabla_x^{LC}y-\frac{1}{2}\theta(x)y-\frac{1}{2}\theta(y)x+
\frac{1}{2}g(x,y)\theta^\sharp\,.
$$
Since $\rho^{Ch}=\rho^B+dJ\theta=\rho^W+dJ\theta$ we just need to prove that
$$
\rho^W=h\,dJ\theta
$$
for some smooth function $h$ on $X$.\\
Since the metric is Vaisman then the Ricci tensor $\text{Ric}^W$ is symmetric and one has (cf. \cite{alexandrov-ivanov-weyl})
$$
\text{Ric}^W(x,x)=\text{Ric}^{LC}(x,x)-
\frac{1}{2}\left(|\theta|^2|x|^2-\theta(x)^2\right)
$$
where $\text{Ric}^{LC}$ is the Ricci tensor of the Levi-Civita connection.\\
Therefore, for the Ricci form of the Weyl connection we can compute
$$
\rho^W(\theta^\sharp,J\theta^\sharp)=
-\text{Ric}^W(\theta^\sharp,\theta^\sharp)=
-\text{Ric}^{LC}(\theta^\sharp,\theta^\sharp)+
\frac{1}{2}\left(|\theta|^2|\theta|^2-\theta(\theta^\sharp)^2\right)=
-\text{Ric}^{LC}(\theta^\sharp,\theta^\sharp)\,.
$$
Since $\theta$ is parallel with respect to the Levi-Civita connection
we have that $R^{LC}(x,y,\theta^\sharp)=0$ for any tangent vectors $x,y$.
Hence, $\text{Ric}^{LC}(\theta^\sharp,\theta^\sharp)=0$ and so
$$
\rho^W(\theta^\sharp,J\theta^\sharp)=0\,.
$$
Now assume, without loss of generality, that $|\theta|=1$ and take a orthonormal basis $\theta,J\theta,\xi,J\xi$.
Then, we have
$$
\omega=\xi\wedge J\xi+\theta\wedge J\theta\,,
\qquad
d J\theta=\xi\wedge J\xi\,,
$$
and
$$
\rho^W=h\xi\wedge J\xi+k\theta\wedge J\theta
$$
for some smooth functions $h,k$. Since, 
$\rho^W(\theta^\sharp,J\theta^\sharp)=0$, then $k=0$ and so
$$
\rho^W=h\xi\wedge J\xi=hdJ\theta.
$$
Now, notice that the scalar curvature $s$ of $g$ is constant if and only if the same holds for the trace $b$ of $\rho^B$. 
This follows by \cite[Formula (2.12)]{alexandrov-ivanov},
$$
b=s-2|\theta|^2+\frac{1}{2}|d\omega|^2
$$
having $|d\omega|^2=|\theta|^2$ constant for a Vaisman metric on a complex surface.
Therefore,  $b$ is constant if and only if  the same holds for the function $h$, indeed
$$
b=\rho^B(J\xi^\sharp,\xi^\sharp)+\rho^B(J(J\xi^\sharp),J\xi^\sharp)+
\rho^B(J\theta^\sharp,\theta^\sharp)+\rho^B(J(J\theta^\sharp),J\theta^\sharp)=-2h\,,
$$
where we have used that $\rho^B=\rho^W=h\xi\wedge J\xi$.
\end{proofof}

\subsection{Pluriclosed Flow on almost abelian Lie groups}

In \cite{arroyo-lafuente} it is shown that the pluriclosed flow on almost abelian Lie algebras reduces to
\[
\left\{\begin{array}{rcl}
            a' &:=&ca \\[5pt]
            v' &:=& cv+Sv-\frac{1}{2}|v|^2v \\[5pt]
            A' &:=& cA
       \end{array}
\right. ;
\]
where
$$
c=\left(\frac{k}{4}-\frac{1}{2}\right)a^2-\frac{1}{2}|v|^2\in\mathbb{R},
\quad
2k=\text{rk}(A+A^t)
$$
and
$$
S=\left(\frac{k}{4}-\frac{1}{2}\right)a^2\text{Id}_{\mathfrak{n}_1}-
\frac{1}{2}AA^t+\frac{a}{4}(A+A^t).
$$
We prove the following
\begin{theorem}
Let $G$ be a simply-connected almost abelian Lie group.
If $A_0 \in\mathfrak{so}(\mathfrak{n}_1)$, then the pluriclosed flow preserves the Bismut K\"ahler-like condition.
\end{theorem}
\begin{proof}
Notice that if $A_0 \in\mathfrak{so}(\mathfrak{n}_1)$ then
by uniqueness of the solution $A(t)\in\mathfrak{so}(\mathfrak{n}_1)$.
Hence, if we assume $A_0 \in\mathfrak{so}(\mathfrak{n}_1)$
one gets
$$
c=-\frac{1}{2}a^2-\frac{1}{2}|v|^2,
\quad
k=0
$$
and
$$
S=-\frac{1}{2}a^2\text{Id}_{\mathfrak{n}_1}+
\frac{1}{2}A^2.
$$
We show that if $A_0 \in\mathfrak{so}(\mathfrak{n}_1)$ and at $t=0$ the metric is Bismut K\"ahler-like then it remains Bismut K\"ahler-like along the flow.
Indeed, if if $A_0 \in\mathfrak{so}(\mathfrak{n}_1)$ and at $t=0$ the metric is Bismut K\"ahler-like one has that $A_0v_0$ is orthogonal to $\mathfrak{n}_1$ and by \cite{arroyo-lafuente} $\mathfrak{n}_1$ is preserved along the flow.
We compute, for any $y\in\mathfrak{n}_1$
$$
\frac{d}{dt}\left(g(v,Ay)\right)=g(v',Ay)+g(v,A'y)=
g(cv+Sv-\frac{1}{2}|v|^2v,Ay)+cg(v,Ay)=
$$
$$
=2cg(v,Ay)-\frac{1}{2}a^2g(v,Ay)+\frac{1}{2}g(A^2v,Ay)-
\frac{1}{2}|v|^2g(v,Ay)=
$$
$$
=3cg(v,Ay)+\frac{1}{2}g(A^2v,Ay)=g\left(v,A(3c\,\text{Id}_{\mathfrak{n}_1}+
\frac{1}{2}A^2)y\right)\,,
$$
hence
$$
\frac{d}{dt}\left(g(Av,y)\right)=
g\left(Av,(3c\,\text{Id}_{\mathfrak{n}_1}+
\frac{1}{2}A^2)y\right)
$$
Then, if $A_0v_0$ is orthogonal to $\mathfrak{n}_1$ by uniqueness of the solution it remains orthogonal along the flow and so the Bismut K\"ahler-like condition is preserved.
\end{proof}

\end{document}